\documentclass[11pt,a4paper,english,reqno]{amsart}
\usepackage{babel}
\usepackage[latin1]{inputenc}
\usepackage{amsmath}
\usepackage{amsthm}
\usepackage{amsfonts}
\usepackage{indentfirst}
\usepackage{graphicx}
\usepackage{mathrsfs}

\usepackage{amssymb}

\newtheorem{teo}{Theorem}[section]
\newtheorem{de}[teo]{Definition}
\newtheorem{pro}[teo]{Proposition}

\newtheorem{rem}[teo]{Remark}
\newtheorem{lem}[teo]{Lemma}
\newtheorem{exam}[teo]{Example}
\newtheorem{alg}[teo]{Algorithm}

\newtheorem*{theorem*}{Theorem}

\newcommand{\gp}{\mathbb{P}}

\renewcommand{\int}{{\rm int}}

\usepackage{color}

\newcommand{\C}{\ensuremath{\mathbb{C}}}
\newcommand{\CP}{\mathbb{C}\mathbb{P}}
\newcommand{\N}{\ensuremath{\mathbb{N}}}
\newcommand{\Q}{\ensuremath{\mathbb{Q}}}

\newcommand{\X}{\ensuremath{\mathbf{X}}}
\newcommand{\la}{\lambda}

\newcommand{\pd}[2]{\frac{\partial #1}{\partial #2}}

\title[On the computation of Darboux first integrals]{On the computation of Darboux first integrals of a class of planar polynomial vector fields}
\author{A.~Ferragut, C.~Galindo \and F.~Monserrat}
\address{A. Ferragut and C. Galindo: Institut Universitari de Matem\`atiques i Aplicacions de Castell\'o (IMAC) and Departament de Matem\`{a}tiques, Universitat Jaume I, Edifici TI (ESTCE), Av. de Vicent Sos Baynat, s/n, Campus del Riu Sec, 12071 Castell\'{o} de la Plana, Spain} \email{ferragut@uji.es, galindo@uji.es}
\address{F. Monserrat: E.T.S. d'Inform\`atica Aplicada, Universitat Polit\`ecnica de Val\`encia, Cam\'\i\ de Vera, s/n, 46002 Val\`encia, Spain} \email{framonde@mat.upv.es}
\date{}
\thanks{Partially supported by the Spanish Government Ministerio de Econom\'ia, Industria y Competitividad (MINECO), grants  MTM2015-65764-C3-2-P, MTM2016-81735-REDT, MTM2016-81932-REDT and MTM2016-77278-P, as well as by Universitat Jaume I grants P1-1B2015-02 and P1-1B2015-16.}

\begin{document}

\begin{abstract}
We study the class of planar polynomial vector fields admitting Darboux first integrals of the type $\prod_{i=1}^r f_i^{\alpha_i}$, where the $\alpha_i$'s are positive real numbers and the $f_i$'s are polynomials defining curves with only one place at infinity. We show that these vector fields have an extended reduction procedure and give an algorithm which, from a part of the extended reduction of the vector field, computes a Darboux first integral for generic exponents.
\end{abstract}

\maketitle

\section{Introduction}
Complex planar polynomial differential systems are being studied since the 19th century when Darboux \cite{dar}, Poincar\'e \cite{poi1,poi2}, Painlev\'e \cite{pai} and Autonne \cite{aut} significantly contributed to this topic. Surprisingly, nowadays, the problem of characterizing integrable differential systems as above remains open. To compute a first integral is a very interesting issue because this function provides the solution curves of the system within their domain of definition, determining the phase portrait of the system.

The Darboux theory of integrability of planar differential systems (as presented for example in \cite{DLA}) considers a remarkable family of multi-valued functions, named Darboux functions, which have the following shape:
\begin{equation}
\label{darboux}
H:= \prod_{i=1}^p f_i^{\lambda_i} \prod_{j=1}^q \exp \left(\frac{h_j}{g_j}\right)^{\mu_j},
\end{equation}
where $f_i$, and $g_j$ and $h_j$ are bivariate complex polynomials and $\lambda_i$ and $\mu_j$ complex numbers. The Darboux theory  states that if a system has $p$ invariant algebraic curves with equations $f_i=0$ with cofactors $k_i$, $i \leq i \leq p$, and $q$ exponential factors $\exp \left(\frac{h_j}{g_j}\right)$ with cofactors $\ell_j$, $1 \leq j\leq q$, such that $\sum_{i=1}^p \lambda_i k_i + \sum_{j=1}^q \mu_j \ell_j =0$ for some complex numbers $\{\lambda_i\}_{i=1}^p$ and $\{\mu_j\}_{j=1}^q$, not all zero, then the function $H$ given in (\ref{darboux}) is a first integral of the differential system.

A particular and desirable type of Darboux functions are rational functions because when $H=f/g$ is a first integral, all the invariant curves of the system are algebraic and are determined from the equations $\lambda f + \beta g=0$, where the pair $(\lambda: \beta)$ runs over the complex projective line. Poincar\'e in \cite{poi2} observed that ``to find out whether a differential equation of the first order and the first degree is algebraically integrable, it is enough to get an upper bound of the degree of the integral". This observation gave rise to the so-called Poincar\'e problem, which looks for a bound of the degree of the first integral in terms of the degree of the polynomial system, and generated a lot of literature.  Although it is well-known that this upper bound does not exist in general, in some cases it can be computed. 
For instance, when the singularities of the corresponding vector field are non-degenerated \cite{poi2}, when they are of nodal type \cite{ce-li} or when the reduction of the before mentioned vector field admits only one non-invariant exceptional divisor \cite{g-m-4}. Notice also that, when a bound of the degree of the rational first integral is known, efficient algorithms to compute that integral have been described \cite{FG2010,bos}.

Recently in \cite{FGM}, the authors considered a family of planar polynomial differential systems $\mathscr{F}$ formed by those systems admitting a polynomial first integral which factorizes as a product of bivariate polynomials, each of them defining a curve with only one place at infinity. These first integrals were called well-behaved at infinity, WAI for short. A plane curve has only one place at infinity when it meets the line at infinity in a unique point where it is reduced and unibranched. Abhyankar and Moh \cite{pa1,pa2,am} introduced these curves. They have a very good local-global behavior, which makes them useful when studying some algebraic and geometric problems \cite{c-p-r-1,c-p-r-2,F-J-ein,F-J-ann,g-m-3}. In addition, many interesting (but hard to compute) tools, introduced for improving the knowledge of the algebraic varieties, are much easier to describe when one considers surfaces having a close relation with curves with only one place at infinity \cite{xie,g-m-adv,mondal,g-m-m}.

Returning to the family $\mathscr{F}$ of differential systems, we proved in \cite{FGM} that when a vector field, or equivalently a $1$-form $\omega$, corresponds to a system $S$ in $\mathscr{F}$, then the degree of the polynomial first integral can be bounded from the knowledge of a part of the Seidenberg reduction of $\omega$. Furthermore, from this reduction, we are able to decide whether $S$ belongs to $\mathscr{F}$ or not and, in the affirmative case, to compute the corresponding first integral. This solves the problem of deciding if a system has a polynomial first integral given by (natural) powers of curves with only one place at infinity.

We want to study whether a similar procedure can be performed for families of polynomial differential systems having a non-polynomial Darboux first integral defined by curves with only one place at infinity. So, in this paper, we consider a new family $\mathscr{D}$ of planar polynomial differential systems having what we call a Darboux positive well-behaved at infinity (DPWAI) first integral. This family satisfies $\mathscr{D} \cap \mathscr{F} = \emptyset$ and roughly speaking (see Definition \ref{DPWAI} for the precise concept) a DPWAI function is a (multi-valued) function $H= \prod_{i=1}^r f_i^{\alpha_i}$, where $f_i$, $1 \leq i \leq r$, are bivariate polynomials defining plane curves $C_i$ with only one place at infinity and satisfying that each one of them does not belong to the pencil at infinity defined by any other. In addition, the values $\alpha_i$, $1 \leq i \leq r$, are positive real numbers satisfying a certain condition which holds when they are linearly independent over the rational numbers.

The main result in this paper is an algorithm whose input is a differential system $\mathbf{X}$ and whose output is either a first integral of $\mathbf{X}$ or ``0''.  The output is a DPWAI first integral when $\mathbf{X}$ belongs to $\mathscr{D}$ and its first integral has generic exponents (see Definition \ref{generic}); if the output is ``0'' then either $\mathbf{X}$ does not belong to $\mathscr{D}$, or it belongs to $\mathscr{D}$ but the exponents of the first integral are not generic. Vector fields admit a (possibly infinite) extended reduction of singularities (see Definition \ref{extendida}). This extended reduction shows that simple singularities whose quotient of eigenvalues is a positive irrational number can be ``simplified" by an infinite sequence of point blowing-ups and represented by means of a proximity graph (see Subsection \ref{proximitygraph} and Section \ref{extendedre}).  Our algorithm uses a (finite) part of that extended reduction and, also,  when $\mathbf{X}$ belongs to $\mathscr{D}$ and its first integral has generic exponents, determines the complete extended reduction over the line at infinity. Notice that, with this algorithm, we are able to compute much more Darboux first integrals (which are not rational) than in \cite{FGM}.


The algorithm has two steps. The first one,  Algorithm \ref{alg1}, uses the mentioned part of the extended reduction to get candidates to polynomials defining the invariant curves $C_i$, and the second one computes the exponents $\alpha_i$ by using Darboux theory of integrability (Theorem \ref{tDar}). We think that, in practice, our algorithm works for any system in $\mathscr{D}$; however, due to our algebraic techniques, we can only guarantee that it computes a Darboux first integral when the exponents $\{\alpha_i\}_{i=1}^r$ are generic.

Section \ref{Alg} provides the mentioned algorithm together with an example showing how it works. The ingredients we need to develop the paper are given in Section \ref{prelim}. WAI first integrals are recalled in Section \ref{Red} and Section \ref{extendedre} introduces the concept of extended reduction of a vector field and describes it for vector fields in $\mathscr{D}$ (generic exponents).

\section{Preliminaires}
\label{prelim}

Let $\X$ be a complex planar polynomial differential system given by
\begin{equation}\label{e1}
\dot x=p(x,y),\quad \dot y=q(x,y),
\end{equation}
where $p,q\in\C[x,y]$ and  $\gcd(p,q)=1$. Assume that $d=\max\{\deg p,\deg q\}$ is the {\it degree} of  the  system $\X$, $\deg$ meaning total degree. In the sequel $\X$  also denotes the corresponding vector field $\X=p\pd{}x+q\pd{}y$.

Recall that a function $H = H(x,y)$ (may be multi-valued) is  a {\it first integral} of $\X$ if $H$ is constant on the solutions of the system. If $H\in\mathcal C^1$ then it satisfies the equation
\[
\X H=p\pd Hx+q\pd Hy=0,
\]
whereas $H$ is defined.

In this paper, we will study a certain family of vector fields $\X$ admitting a particular class of Darboux first integrals. We devote this section to summarize some concepts and properties we will need.

We have introduced the polynomial differential system (\ref{e1})  by using affine coordinates, however, in this paper, we will need to consider its complex projectivization. Therefore we start by studing the complex projectivization $\mathcal X$ of the vector fields $\X$ and their corresponding $1$-forms.

\subsection{Polynomial vector fields in $\CP^2$}

Let $A$, $B$, and $C$ be homogeneous polynomials of degree $d+1$ in the  variables $X$, $Y$, and $Z$ with coefficients in the complex numbers. A homogeneous $1$-form
\[
\Omega=AdX + BdY + CdZ
\]
of degree $d+1$ is called {\it projective} if $XA+YB+ZC=0$. This means that there exist three homogeneous polynomials $P$, $Q$, and $R$, in the variables $X, Y, Z$, of degree $d$ such that
\[
A=ZQ-YR, \; \; B=XR-ZP, \; \; C=YP-XQ.
\]
Then, we can write
\begin{equation}\label{omega}
\Omega = P(YdZ - ZdY) + Q(ZdX - XdZ) + R( XdY - YdX).
\end{equation}
The $1$-form $\Omega$ is usually called a Pfaff algebraic form of the complex projective plane $\CP^2$ \cite{jou}, and the triple $(P, Q, R)$ determines a homogeneous polynomial vector field in $\CP^2$ of degree $d$:
\[
\mathcal X=P\pd{}X +Q\pd{}Y +R\pd{}Z.
\]

Let $F\in\C[X,Y, Z]$ be a homogeneous polynomial. The curve  $F=0$ in $\CP^2$  is {\it invariant} under the flow of the vector field $\mathcal X$ if
\begin{equation}\label{XF}
\mathcal XF =P\pd{F}X +Q\pd{F}Y +R\pd{F}Z= KF,
\end{equation}
for some homogeneous polynomial $K \in C[ X, Y, Z]$ of degree $d - 1$, called the {\it cofactor} of $F$.

The singular points of a projective $1$-form $\Omega$ of degree $d+ 1$ or of its associated homogeneous polynomial vector field $\mathcal X$ of degree $d$ are those points in the projective plane satisfying the following system of equations:
\begin{equation}\label{SP}
ZQ-YR=0,\quad XR-ZP=0,\quad YP-XQ=0.
\end{equation}

Next we show how to get the projectivization, as a $1$-form or as a vector field, of System (\ref{e1}). System \eqref{e1} is equivalent to the $1$-form
\[
p(x, y) dy-q(x, y) dx,
\]
which extends to $\CP^2$ as the projective $1$-form of degree $d+ 1$
\begin{equation}\label{Eq.1form}
Z^{d+2}\left(p\left(\frac XZ,\frac YZ\right)\frac{YdZ-ZdY}{Z^2}-q\left(\frac XZ,\frac YZ\right)\frac{XdZ-ZdX}{Z^2}\right),
\end{equation}
where we have replaced $(x,y)$ by $(X/Z,Y/Z)$. Set $P(X,Y,Z)=Z^dp(X/Z,Y/Z)$ and $Q(X,Y,Z)=Z^dq(X/Z,Y/Z)$, then the $1$-form \eqref{Eq.1form} becomes
\[
P ( X , Y , Z )( YdZ - ZdY )+ Q ( X , Y , Z) ( ZdX - XdZ ).
\]
As a consequence, the vector field $\X$  is extended to  the homogeneous polynomial vector field of degree $d$ in $\CP^2$
\begin{equation}\label{Xproj}
\mathcal X = P \pd{}X + Q \pd{}Y.
\end{equation}
which is called the {\it complex projectivization} of the vector field $\X$. Notice that the line at infinity $Z = 0$ is a solution of the projective vector field because the third component $R$ in $\mathcal X$ is identically zero.
In addition, by \eqref{SP}, it turns out that the singular points of the complex projectivization of  System \eqref{e1} satisfy the following equalities:
\[
ZQ( X, Y, Z) = 0,\quad ZP( X, Y, Z) = 0, \quad YP( X, Y, Z) - XQ( X, Y, Z) = 0. \]
The third equation and the line $Z=0$ determine the singular points at infinity. Setting $Z=1$, the singular points which are not at infinity are obtained from the equalities $P=Q=0$.

Recall that an {\it invariant} algebraic curve of the vector field $\X$ is an affine algebraic curve with local equation $f(x,y)=0$, $f\in \mathbb{C}[x,y]$, such that $\X f = kf$, where $k\in \mathbb{C}[x,y]$ is called the cofactor of $f=0$ and has degree at most $d-1$. Now, if $f( x, y) = 0$ has degree $n\in\N$, then $F( X, Y, Z) = Z^nf(X/Z,Y/Z)=0$ is an invariant algebraic curve of the vector field \eqref{Xproj} with cofactor $K(X,Y, Z)=Z^{d-1}k(X/Z,Y/Z)$.

We have briefly reviewed the projectivization procedure for affine systems. We conclude this subsection by explaining how the affinization works. We will consider the local chart determined by $Z = 1$. Let $F = 0$ be an invariant algebraic curve of degree $n$ of the vector field defined by \eqref{omega} with cofactor $K$. Applying Euler's Theorem for homogeneous functions  and taking \eqref{XF} into account, it holds that $f(x,y)=F (X, Y, 1) = 0$ is an equation of an invariant algebraic curve of the restriction of $\Omega$ to the affine plane:
\[
\left(P(x,y,1)-xR(x,y,1)\right)dy - \left(Q(x,y,1)-yR(x,y,1)\right)dx.
\]
This $1$-form has degree $d+1$ and the cofactor of the invariant curve $f(x,y)=0$ is $k(x,y) = K(x,y,1)-nR(x,y,1)$. It has degree at most $d$ whenever the line $Z = 0$ is not invariant. We notice that the line $Z=0$ is invariant if and only if $Z|R$. In such a case, the degree of $k$ is at most $d-1$.

Next, we give a short overview of the blow-up technique to reduce the singularities of a planar vector field, which will be a key element in this paper.

\subsection{Reduction of singularities}

The singularities of planar vector fields can be reduced by blowing-up (see \cite[Section 4.1]{FGM} for details). This procedure, due to Seidenberg,  performs algebraic modifications and gives rise to a simpler vector field over a different to $\mathbb{CP}^2$ surface, which makes easier the study of the original vector field \cite{seid,D,AFJ}. Next we recall the concepts of singularity and simple singularity.

\begin{de}
{\rm A point $O\in \mathbb{C}^2$ is called a \emph{singularity} of a polynomial vector field $\X=p(x,y)\pd{}x+q(x,y)\pd{}y$ in $\mathbb{C}^2$, $\{x,y\}$ being local coordinates at $O$, if the multiplicity of $\X$ at $O$ (that is, the minimum of the orders of $p=p(x,y)$ and $q=q(x,y)$ at $O$) is strictly positive. Moreover $O$ is called a \emph{simple} singularity whenever $\X$ has multiplicity $1$ at $O$ and the matrix
$$\begin{pmatrix} \pd{p_1}x & \pd{p_1}y\\ \pd{q_1}x & \pd{q_1}y \end{pmatrix}, $$
defined by the first nonzero jet $p_1 dy - q_1dx$ of the differential $1$-form $p dy - qdx$,
has eigenvalues $\lambda_1,\lambda_2$ satisfying either $\lambda_1\lambda_2\not=0$ and $\frac{\lambda_1}{\lambda_2}\not\in\Q^+$, or $\lambda_1\lambda_2=0$ and $\lambda_1^2+\lambda_2^2\neq0$.
An \emph{ordinary singularity} is a singularity that is not simple. }
\end{de}

Seidenberg's reduction theorem \cite{seid} (see also \cite{Brunella}) shows that the singularities of a planar vector field can be transformed into simple singularities by means of blowing-ups:

\begin{teo}
    Let $O\in \mathbb{C}^2$ be an isolated singularity of a polynomial vector field  $\X=p\pd{}x+q\pd{}y$ in $\mathbb{C}^2$. Then there exists a finite sequence of blowing-ups centered at closed points of the successively obtained surfaces, $\pi:\mathcal Z\rightarrow \mathbb{CP}^2$, such that all the singularities of the strict transform of the vector field at the surface $\mathcal Z$ are simple.
\end{teo}

From a global point of view, Seidenberg's result states that, given an homogeneous polynomial vector field $\mathcal{X}$ in $\mathbb{CP}^2$, there exists a set of points (or \emph{configuration}, according with a forthcoming definition)
\[
{\mathcal S}(\mathcal{X})=\{Q_0, Q_1, \ldots , Q_n\}
\]
such that $Q_0 \in {X}_0 := \mathbb{CP}^2$, $\pi_{Q_{i-1}}: \mathrm{Bl}_{Q_{i-1}} (X_{i-1})\rightarrow X_{i-1}$ is the blowing-up of $X_{i-1}$ centered at $Q_{i-1}$,
 $ Q_i \in \mathrm{Bl}_{Q_{i-1}}( X_{i-1}) =: X_{i}$ for $1 \leq i \leq n$, and the composition
 $$\pi_{Q_0}\circ\cdots \circ \pi_{Q_n}: \mathcal Z:=X_{n+1}\longrightarrow \mathbb{CP}^2$$
is such that all the singularities of the strict transform $\tilde{\mathcal{X}}$ of $\mathcal{X}$ in $\mathcal Z$ are simple. We call ${\mathcal S}(\mathcal{X})$ the \emph{singular configuration} of $\mathcal{X}$. For an explicit local description of the reduction of singularities and the obtention of the singular configuration, including a detailed example, see Section 4 of \cite{FGM}.

A \emph{solution} of $\mathcal{X}$ at a point $p \in \mathbb{CP}^2$ is a holomorphic (possibly singular) irreducible curve on a neighborhood of $P$ which passes through $p$ and is invariant by $\mathcal{X}$.  A singularity is called \emph{dicritical} if there are infinitely many solutions through it (see \cite[Definition 1]{FGM} for an equivalent definition that allows us to detect the dicritical character of a singularity from the reduction process).

Along the paper, we will denote by $\mathcal{D}(\mathcal{X})$ the \emph{dicritical configuration} of $\mathcal{X}$, which is  the set of points $Q_i$ in $\mathcal{S}(\mathcal{X})$ such that $Q_i$ is a dicritical singularity of the strict transform of the vector field $\mathcal{X}$ at $Q_i$ (see \cite[page 360]{FGM} for the definition of strict transform of a vector field and the straightforward translation of the concept of dicritical singularity to these vector fields over surfaces obtained by blowing-up).\\

In the following two subsections we recall the concepts of configuration of infinitely near points and its proximity graph, and that of proximity graph defined by a positive real number.

\subsection{Proximity graph of a configuration: singular and dicritical graphs}

Let $P$ be a point of a complex surface. The {exceptional divisor} $E_P$  produced by blowing-up $P$ is called the {\it first infinitesimal neighborhood} of $P$. By induction, if $i>0$, then the points in the $i$th infinitesimal neighborhood of $P$ are the points in the first infinitesimal neighborhood of some point in the $(i-1)$th infinitesimal neighborhood of $P$. A point $Q\neq P$ in some infinitesimal neighborhood of $P$ is called \emph{proximate} to $P$ if $Q$ belongs to the strict transform of $E_P$. Also $Q$ is a \emph{satellite} point if it is proximate to two points; that is, if it is the intersection point of the strict transforms of two exceptional divisors. Non-satellite points are named \emph{free}.

Points in the $i$th infinitesimal neighborhood of $P$, for some $i>0$, are {\it infinitely near} to $P$. These points admit a natural ordering: a point $R$ precedes $Q$ if and only if $Q$ is infinitely near to $R$. Note that we agree that a point is infinitely near to itself.

A \emph{configuration of infinitely near points} (or, simply, a configuration) of a complex surface $X_0$ (it could be $\mathbb{CP}^2$ ) is a (finite or infinite) set of points
\[
{\mathcal C}=\{P_0, P_1,\ldots\},
\]
such that $P_0\in X_0$ and, for all $i\geq 1$, $P_i$ belongs to the blow-up $X_i$ of $X_{i-1}$ with center at $P_{i-1}$.

The Hasse diagram of $\mathcal C$, $H_{\mathcal C}$, with respect to the above alluded order relation is a union of rooted trees  whose set of vertices is bijective with $\mathcal C$. The {\it proximity graph} of $\mathcal{C}$, $\Gamma_{\mathcal C}$, is a labeled graph formed by $H_{\mathcal C}$ where, also, we
join with a dotted edge those vertices corresponding with points $P$ and $Q$ of $\mathcal C$ such that $Q$ is proximate to $P$ but $Q$ is not in the first infinitesimal neighborhood of $P$. For simplicity we delete those dotted edges that can be deduced from others.

If $\mathcal{X}$ is, as above, a projective vector field over $\mathbb{CP}^2$,  the proximity graph $\Gamma_{{\mathcal S}(\mathcal X)}$ (respectively, $\Gamma_{{\mathcal D}(\mathcal X)}$) is called the \emph{singular graph} (respectively, \emph{dicritical graph}) of $\mathcal{X}$.

\subsection{The proximity graph defined by a positive real number}\label{proximitygraph}

Let $X_0$ be a complex surface, $P$ a point in $X_0$ and $C$  a germ of curve on the local ring of $X_0$ at $P$ having only one analytic branch. Assuming that $P$ is singular, one can determine the configuration of infinitely near points ${\mathcal D}_C=\{P_0,P_1,\ldots, P_s \}$ such that:
\begin{itemize}
\item[(i)] $P_0=P$ and, for all $i\geq 1$, $P_i$ is the point where the exceptional divisor $E_{P_{i-1}}$  meets the strict transform of $C$.
\item[(ii)] The composition of the blowing-ups centered at the points of ${\mathcal D}_C$ gives rise to a minimal embedded resolution of the singularity of $C$ at $P$.
\end{itemize}
For $i \geq 0$, let $m_i$ denote the multiplicity at $P_i$ of the strict transform of $C$. Each one of these numbers $m_i$ satisfies the so-called \emph{proximity equalities}: $m_i=1$ if $i=s$ and, otherwise, $m_i=\sum m_j$, where the sum runs over the set of indices $j$ such that $P_j$ is proximate to $P_i$ \cite{C}. Notice that the $(s+1)$-tuple of multiplicities $(m_0,m_1,\ldots,m_s)$ uniquely determines the proximity graph of the configuration ${\mathcal D}_C$.

If the singularity of $C$ at $P$ has only one Puiseux pair (i.e., the minimal embedded resolution is obtained by blowing-up some free points and, afterwards, finitely many satellite points), then the sequence of multiplicities of $C$ at $P$ has the shape $$({r_0}_{(c_0)}, {r_1}_{(c_1)},\ldots,{r_{\ell}}_{(c_{\ell})}={1}_{(c_{\ell})}),$$  where the subindices $(c_i)_{i=0}^{\ell}$ indicate the number of times that each multiplicity is repeated. Moreover, the numbers $c_i$ come from the continued fraction expansion of the rational number
$$\frac{\sum_{i=0}^s m_i^2}{m_0^2}=[c_0;c_1,\ldots,c_{\ell}]:=c_0+\frac{1}{c_1+\frac{1}{c_2+\ldots \frac{1}{c_{\ell}}}};$$
$r_0=m_0$ and $r_1, r_2, \ldots, r_{\ell}$ are the successive remainders appearing when the Euclidean algorithm is applied to $\sum_{i=0}^s m_i^2$ and $m_0^2$ \cite{C}. Hence, in this case, the proximity graph of ${\mathcal D}_C$ is determined by the rational number $\frac{\sum_{i=0}^s m_i^2}{m_0^2}$. Similarly, the continued fraction expansion of any positive rational number $\beta$ determines a proximity graph, which we shall denote by ${\bf Prox}(\beta)$.

Now, let $\gamma$ be a positive irrational number and consider its (infinite) continued fraction expansion
$$\gamma=[c_0;c_1, c_2, \ldots]:=c_0+\frac{1}{c_1+\frac{1}{c_2+\ldots}}.$$
For each $i\geq 0$ set $\beta_i:=[c_0; c_1, \ldots, c_i]$. Then, the proximity graph defined by $\gamma$, ${\bf Prox}(\gamma)$, is given by the limit of the graphs ${\bf Prox}(\beta_i)$ in the following sense: the vertices (respectively, edges) of ${\bf Prox}(\gamma)$ are the vertices (respectively, edges) of ${\bf Prox}(\beta_i)$ for $i$ large enough. It has the shape described in Figure \ref{grafinf}.

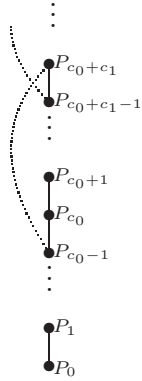
\begin{figure}[ht!]
\centering
\setlength{\unitlength}{0.5cm}
\begin{picture}(26,12)
\qbezier[50](11,4)(9,7)(11,9)
\qbezier[20](11,8)(10,9)(10,10)


\put(11,1){\circle*{0.3}}\put(11.1,0.8){\tiny{$P_0$}}
\put(11,1){\line(0,1){1}}

\put(11,2){\circle*{0.3}}\put(11.1,1.9){\tiny{$P_1$}}

\put(10.9,3){$\vdots$}

\put(11,4){\circle*{0.3}}\put(11.1,3.9){\tiny{$P_{c_0-1}$}}
\put(11,4){\line(0,1){1}}

\put(11,5){\circle*{0.3}}\put(11.1,4.9){\tiny{$P_{c_0 }$}}
\put(11,5){\line(0,1){1}}

\put(11,6){\circle*{0.3}}\put(11.1,5.9){\tiny{$P_{c_0 +1}$}}

\put(10.9,7){$\vdots$}

\put(11,8){\circle*{0.3}}\put(11.1,7.9){\tiny{$P_{c_0 +c_1-1}$}}
\put(11,8){\line(0,1){1}}

\put(11,9){\circle*{0.3}}\put(11.1,8.9){\tiny{$P_{c_0 +c_1}$}}

\put(11,10){$\vdots$}

\end{picture}

\caption{Proximity graph defined by $\gamma$.}
\label{grafinf}
\end{figure}

The relation between elimination of base points of certain linear systems and the reduction of singularities of planar vector fields having a rational first integral supports some reasonings in this paper. So we conclude this section with a brief of those concepts close to linear systems we will use.

\subsection{Linear systems and pencils}

A \emph{linear system} on $\mathbb{CP}^2$ is the set of algebraic curves defined by the polynomials in a subspace of $\mathbb{C}_m[X,Y,Z]$ for some natural number $m>0$, where $\mathbb{C}_m[X,Y,Z]$ denotes the (projective) space of homogeneous polynomials of degree $m$ in the variables $X,Y,Z$. Notice that a linear system has structure of projective space. A \emph{pencil} is a linear system of projective dimension $1$. Next, we introduce the concept of cluster of $\mathbb{CP}^2$.

\begin{de}  {\rm
A \emph{cluster} of infinitely near points (or, simply,  a cluster) of $\mathbb{CP}^2$ is a pair $({\mathcal C}, {\mathbf m})$ where $\mathcal C=\{Q_0, Q_1, \ldots,Q_h\}$ is a configuration of infinitely near points of $\mathbb{CP}^2$ and ${\mathbf m}=(m_0, m_1, \ldots,m_h)\in \N^n$, $\N$ being the set of positive integers.}
\end{de}

We desire to consider linear systems determined by clusters. To this purpose, we require some notations and concepts. We start by recalling the concept of virtual transform (respectively, passing virtually) at (respectively, through) a point of a cluster. For each $Q_i\in {\mathcal C}$, we set
\[
\ell(Q_i) := \mathrm{card} \{Q_j\in {\mathcal C} | \mbox{ $Q_i$ is infinitely near to $Q_j$} \}.
\]
Now consider a cluster ${\mathcal K}=({\mathcal C}, {\mathbf m})$, an algebraic curve $C$ in $\mathbb{CP}^2$, and a point $Q_k\in {\mathcal C}$. When $\ell(Q_k)=1$, we take a local chart at $Q_k$ with local coordinates $(x,y)$ and a local equation of $C$, $f(x,y)=0$. The \emph{virtual transform} of $C$ at $Q_k$ with respect to the cluster $\mathcal K$ (denoted by $C^{\mathcal K}_{Q_k}$) is the (local) curve defined by $f(x,y)=0$. The multiplicity of $C^{\mathcal K}_{Q_k}$ at $Q_k$, denoted $m_{Q_k}(C^{\mathcal K}_{Q_k})$, is the degree of the first non-zero jet of $f(x,y)$. Finally, we say that $C$ \emph{passes virtually} through $Q_k$ with respect to $\mathcal K$ whenever $ m_{Q_k}(C^{\mathcal K}_{Q_k}) \geq m_k$.

Suppose now that $\ell(Q_k)>1$ and set $Q_j\in {\mathcal C}$ such that $Q_k$ is in the first infinitesimal neighborhood  of $Q_j$. Assume inductively that $C$ \emph{passes virtually} through $Q_j$ with respect to $\mathcal K$, take local coordinates $(x,y)$ at $Q_j$ and let $f(x,y)=0$ be a local equation of $C^{\mathcal K}_{Q_j}$. The point $Q_k$ belongs to the surface obtained by blowing-up $Q_j$ and thus
$Q_k=(0,\lambda)$ (respectively, $Q_k=(\lambda,0)$) in local coordinates $(x,t=y/x)$ (respectively, $(s=x/y,y)$). This allows us to define the \emph{virtual transform} of $C$ at $Q_k$ with respect to the cluster $\mathcal K$, $C^{\mathcal K}_{Q_k}$, as the (local) curve defined by $x^{-m_j}f\left( x,x(t+\lambda) \right)=0$ (respectively, $x^{-m_j}f\left((s+\lambda)y,y\right)=0$). As above, the multiplicity of $C^{\mathcal K}_{Q_k}$ at $Q_k$ is denoted by $m_{Q_k}(C^{\mathcal K}_{Q_k})$ and $C$ \emph{passes virtually} through $Q_k$ with respect to $\mathcal K$ if $m_{Q_k}(C^{\mathcal K}_{Q_k}) \geq m_k$. When the curve $C$ passes virtually through $Q_i$ with respect to $\mathcal K$ for all $Q_i\in {\mathcal K}$, we say that $C$ \emph{passes virtually} through $\mathcal K$.\\

Next we define the above mentioned linear system determined by a cluster and a positive integer.

\begin{de}
{\rm
Given a positive integer $m$ and a cluster $\mathcal K=({\mathcal C},\mathbf{m})$ of $\mathbb{CP}^2$, the \emph{linear system determined by $m$ and $\mathcal K$}, denoted by
${\mathcal L}_m(\mathcal K)$ or ${\mathcal L}_m({\mathcal C},\mathbf{m})$, is the linear system in $\mathbb{CP}^2$ given by those curves defined by polynomials in $\mathbb{C}_m[X,Y,Z]$ that pass virtually through $\mathcal K$.}
\end{de}

Another important concept for us is that of the strict transform of a curve on a surface obtained by a sequence of point blowing-ups.

\begin{de}
{\rm
The {\it strict transform} $\tilde{C}$ of an algebraic curve $C$ on a complex surface $Z$ obtained by the composition $\pi:Z\rightarrow \mathbb{CP}^2$ of the sequence of point blowing-ups associated with a (finite) configuration of infinitely near points ${\mathcal C}$ is the image of $C$ by the birational map $\pi^{-1}$.
}
\end{de}

Let $n$ be a positive integer, let $F_1, F_2, \ldots, F_s\in \mathbb{C}_n[X,Y,Z]$ be linearly independent polynomials in $\mathbb{C}_n[X,Y,Z]$ and assume that $F_1, F_2,\ldots,F_s$ have no common factor. Set $\mathcal L= \mathbb{P}V$ the linear system on $\mathbb{CP}^2$ associated to the linear space over $\mathbb{C}$, $V=\langle F_1, F_2, \ldots,F_s\rangle$, spanned by the polynomials $F_i$, $1 \leq i \leq s$. Then, there exists a configuration of infinitely near points of $\mathbb{CP}^2$, $\mathcal{BP}({\mathcal L})$, and a finite set of linear subspaces ${\mathcal H}_i\subsetneq \mathbb{CP}^{s-1}$, $1\leq i\leq t$, such that the strict transforms of the curves with equations $$\alpha_1 F_1(X,Y,Z)+ \alpha_2 F_2(X,Y,Z)+ \cdots+\alpha_s F_s(X,Y,Z)=0,$$  $(\alpha_1, \alpha_2, \ldots,\alpha_s)\in \mathbb{CP}^{s-1}\setminus \bigcup_{i=1}^t {\mathcal H}_i$ (called  \emph{generic} curves of $\mathcal L$) have the same multiplicities at every point  $Q\in \mathcal{BP}({\mathcal L})$ (denoted by $\mathrm{mult}_Q({\mathcal L})$) and have pairwise empty intersections at the surface obtained by blowing-up the points in $\mathcal{BP}({\mathcal L})$.  Notice that, if ${\mathcal L}$ is a pencil, then $\bigcup_{i=1}^t {\mathcal H}_i$ is a finite set.

\begin{de} {\rm
 With the above notations, the pair $(\mathcal{BP}({\mathcal L}),\mathbf{m})$, where $$\mathbf{m}=\left(\mathrm{mult}_Q({\mathcal L})\right)_{Q\in \mathcal{BP}({\mathcal L})},$$ is a cluster of inifinitely near points of  $\mathbb{CP}^2$ called the \emph{cluster of base points} of $\mathcal L$.}
\end{de}

\section{Reduction of singularities of a vector field with WAI first integral}
\label{Red}
We start this section by recalling the concept of plane curve with only one place at infinity, which was initially considered in \cite{pa1,pa2,am}.

\begin{de}
{\rm An algebraic projective curve $C$ defined by a homogeneous polynomial  $F\in\C[X,Y,Z]$ has {\it only one place at infinity} if $C$ meets the line at infinity $Z=0$ at a unique point $P$ and $C$ is reduced and unibranched (i.e., analytically irreducible) at $P$.
}
\end{de}

Next we introduce the concept of well-behaved at infinity (WAI) first integral.

\begin{de}
{\rm
A complex planar polynomial differential system (or vector field) $\X$ has a WAI first integral if it has a polynomial first integral of the form
\[
H=\prod_{i=1}^r f_i^{n_i},
\]
where $r$ and $n_i$, $1 \leq i \leq r$, are positive integers and $f_i$, $1 \leq i \leq r$, are polynomials in $\C[x,y]$ of degree $d_i\in\N$ such that each complex projective curve $C_i$ of $\mathbb{CP}^2$ defined by the projectivization $F_i(X,Y,Z)=Z^{d_i}f_i(X/Z,Y/Z)$ of $f_i$ has only one place at infinity.

Along this paper we will also assume that, for having a WAI first integral $H$,
$$\gcd(n_1, n_2, \ldots,n_r)=1,\;\;\;r \geq 2,$$ and that $H$ fulfills the following condition:
\begin{equation}
\label{S}
f_i-\lambda f_j\not\in \mathbb{C} \;\mbox{ for all $i,j\in \{1, 2, \ldots,r\}$ such that $i\neq j$ and for all $\lambda \in \mathbb{C}$.}
\end{equation}
}
\end{de}
\vspace{2mm}

For $1 \leq j \neq i \leq r$, write
\begin{equation}
\label{ro}
\rho_{ji} := \sum_Q (C_j,C_i)_Q,
\end{equation}
where $Q$ runs over the set of common points of $C_j$ and $C_i$ outside the line at infinity (defined by $Z=0$), and $(C_j,C_i)_Q$ is the intersection multiplicity between $C_i$ and $C_j$ at $Q$.

We will assume in the remaining of this section that $\X$ is a complex planar polynomial vector field which admits a WAI first integral. By \cite[Corollary 1]{FGM}, the dicritical configuration $\mathcal{D}(\mathcal{X})$ of the complex projectivization $\mathcal{X}$ of $\X$ to $\mathbb{CP}^2$ coincides with the configuration of the cluster of base points of the pencil $\mathbb{P}V$, where $V = \langle \prod_{i=1}^r F_i^{n_i}, Z^d\rangle$ and $d=\sum_{i=1}^r n_i d_i$, $d_i:=\deg(F_i)$. This cluster is studied in  \cite{c-p-r-2}. Taking advantage of it we state the following result, which provides a very specific information about the dicritical configuration $\mathcal{D}(\mathcal{X})$.

\begin{pro}\label{previa}
Let $\X$ and $\mathcal{X}$ be as above. Then
$$\mathcal{D}(\mathcal{X})=\bigcup_{i=1}^r  \left(\mathcal{BP}({\mathcal P}_i)\cup \{S_i\}\cup {\mathcal L}_i\right),$$
where, for each $i=1, 2, \ldots,r$:
\begin{itemize}
\item[(a)] $\mathcal{BP}({\mathcal P}_i)$ is the configuration of base points of the pencil ${\mathcal P}_i$ defined by the non-zero linear combinations of the polynomials $F_i$ and $Z^{d_i}$.

\item[(b)] The configuration $\mathcal{BP}({\mathcal P}_i)$, ordered by the relation ``to be infinitely near to", has only one  maximal point $Q_i$. Moreover, the strict transform of the curve $C_i$ with equation $F_i=0$ meets the exceptional divisor $E_{Q_i}$ at a unique point $S_i$. The local 1-form at $S_i$ defining the strict transform of $\mathcal{X}$ has the shape $\delta_i u\;dt-n_i t\;du$, where $u=0$ (respectively, $t=0$) is a local equation of the strict transform of $C_i$ at $S_i$ (respectively, the exceptional divisor $E_{Q_i}$) and
$$\delta_i=\sum_{\substack{j=1 \\ j \neq i}}^r n_j\rho_{ji}.$$

Furthermore, $\{S_i\}\cup {\mathcal L}_i$ is the configuration of base points of the (local) pencil (at $S_i$) defined by the non-zero linear combinations of $u^{n_i}$ and $t^{\delta_i}$.

\item[(c)] No value $\rho_{ji}$ equals zero.

\item[(d)] The point $S_i$ is free.

\item[(e)] If $i\neq j$ then $S_i$ is not infinitely near to $S_j$ and $S_j$ is not infinitely near to $S_i$.

\end{itemize}

\end{pro}

\begin{proof}
Items (a), (b), and (e) follow from \cite[Lemma 1]{c-p-r-2} and its proof. The composition of the blowing-ups centered at the points of ${\mathcal P}_i$ is an embedded resolution of the singularity of $C_i$ at infinity (see \cite{pa2, Mo}) and, as a consequence, the point $S_i$ is free. This proves Item (d).

It remains to prove (c). Reasoning by contradiction, assume that $\rho_{ji}=0$ for some indexes $i, j \in \{1, 2, \ldots,r\}$ such that $i\neq j$. Let $Y$ be the surface obtained by blowing-up the points in $\mathcal{BP}({\mathcal P}_i)$ and let $\tilde{C}_i$ and $\tilde{C}_j$ be the respective strict transforms of $C_i$ and $C_j$ on $Y$. By (b) and (e) and our assumption,   $\tilde{C}_i$ and $\tilde{C}_j$ do not meet, and then $\tilde{C}_i\cdot \tilde{C}_j=0$. This means, by the projection formula, that $\tilde{C}_j$ is contracted by the morphism defined by a basis of global sections of the sheaf $\mathcal O_{Y}(\tilde{C}_i)$  \cite{hart}. Thus $C_j$ must be a curve of the pencil $\mathcal{P}_i$, and this contradicts Condition (\ref{S}), which is assumed for the WAI first integral $H$.
\end{proof}

 \begin{rem}
 \label{NOTA}
 {\rm The configuration $\bigcup_{i=1}^r  \left(\mathcal{BP}({\mathcal P}_i)\cup \{S_i\}\right)$
is independent of the exponents $n_1, n_2, \ldots,n_r$ appearing in the WAI first integral.
}
\end{rem}

\section{Extended reduction of singularities and DPWAI first integrals}
\label{extendedre}

An algorithm for deciding whether a polynomial differential system as (\ref{e1}) has a WAI first integral, and computing that integral in the affirmative case, was given in \cite{FGM}. This algorithm uses a part of the reduction of the corresponding vector field $\mathcal{X}$. We want to find out whether there is some close procedure to compute other types of Darboux first integrals involving curves with only one place at infinity. Next we introduce a class of Darboux functions and an extended reduction of singularities of vector fields suitable for our purposes.
\begin{de}
\label{DPWAI}
{\rm
A Darboux positive well-behaved at infinity (DPWAI, for short) function is a multi-valued function of the form
\[
H=\prod_{i=1}^rf_i^{\alpha_i}, \;\; r \geq 2,
\]
where:
\begin{enumerate}
\item The $f_i$'s, $i \leq i \leq r$, are polynomials in $\C[x,y]$ of degree $d_i\in\N$ such that the curve $C_i$ defined by the projectivization $F_i(X,Y,Z)=Z^{d_i}f_i(X/Z,Y/Z)$ of $f_i$ has only one place at infinity.
    \item The values $\alpha_i$, $i \leq i \leq r$, are strictly positive real numbers.
    \item The polynomials $f_i$, $1 \leq i \leq r$ satisfy Condition (\ref{S}).
    \item For all $i\in \{1, 2, \ldots,n\}$, there is no positive rational number $\beta$ such that $$\beta \alpha_i=\sum_{\substack{j=0 \\ j \neq i}}^r \rho_{ji}\alpha_j,$$
        where $\rho_{ji}$ is the value defined in (\ref{ro}).
\end{enumerate}
}
\end{de}

\begin{rem}
{\rm Note that our last condition  holds, in particular, when $\{\alpha_i\}_{i=1}^r$ is a linearly independent set over the field $\mathbb{Q}$.
}
\end{rem}

Now, we define the concept of extended reduction of singularities.

\begin{de}
\label{extendida}
{\rm
An \emph{extended reduction of singularities} of an arbitrary singular polynomial vector field ${\mathcal X}$ on $\mathbb{CP}^2$ is a sequence of blowing-ups
\begin{equation}\label{ext}
\cdots \rightarrow X_{i+1}\rightarrow X_i\rightarrow \cdots \rightarrow X_1\rightarrow X_0=\mathbb{CP}^2
\end{equation}
obtained by performing, first, a reduction of singularities of ${\mathcal X}$ and, then, by successively blowing-up every simple singularity of the transformed vector field whose quotient of eigenvalues is a positive real number.
}
\end{de}

This extended reduction can be regarded as a more natural procedure than Seidenberg's reduction because we keep blowing-up points while the quotient of eigenvalues is a positive real number $\gamma$. Although we can obtain an infinite sequence of blowing-ups, this sequence is completely determined by the continued fraction expansion of $\gamma$. Seidenberg applies only this procedure when $\gamma$ is rational giving rise to a finite sequence of point blowing-ups.

Let us denote by ${\mathcal E}(\mathcal{X})$ the configuration of infinitely near points of $\mathbb{CP}^2$ formed by the centers of the extended reduction of singularities of $\mathcal{X}$. Also, ${\mathcal E}_{\infty}(\mathcal{X})$ will denote the configuration of points in $\mathcal{E}(\mathcal X)$ whose images on $\mathbb{CP}^2$ by the sequence (\ref{ext}) belong to the line at infinity.

\begin{de}\label{generic}
{\rm

We will say that a property $\bf P$ is satisfied for \emph{generic} exponents $\{\alpha_i\}_{i=1}^r$ if there exists a finite set of non-zero polynomials $\{h_j(z_1, z_2, \ldots,z_r)\}_{j\in J}\subseteq \mathbb{C}[z_1,\ldots,z_r]$, $J$ a set of indexes, fulfilling the following condition: $\bf P$ is satisfied for all $r$-tuples $(\alpha_1, \alpha_2, \ldots,\alpha_r)$ such that $h_j(\alpha_1,\ldots,\alpha_r)\neq 0$ for all $j\in J$.
}
\end{de}

Until the end of this section, we will suppose that $\X$ is a polynomial vector field (or differential system) of $\mathbb{C}^2$ having a DPWAI first integral
$$
H=\prod_{i=1}^rf_i^{\alpha_i}.
$$

Consider the complex projectivization ${\mathcal{X}}$ of $\X$ and set $\Omega_{\mathcal{X}}=AdX+BdY+CdZ$ a homogeneous reduced 1-form defining ${\mathcal{X}}$. The following result explains how $\Omega_{\mathcal{X}}$ writes for generic exponents. See \cite{LP2004} for the affine version with a polynomial differential system having a generalized Darboux first integral.

\begin{pro}\label{prop}
With the above notation it follows that, for generic exponents $\alpha_1, \ldots,\alpha_r$, the reduced $1$-form $\Omega_{\mathcal{X}}$ is, up to multiplication by a non-zero constant, equal to
\begin{multline*}
\left(\sum_{i=1}^r \alpha_i Z \prod_{\substack{j=1 \\ j \neq i}}^r F_j \frac{\partial F_i}{\partial X}\right) dX+\left(\sum_{i=1}^r \alpha_i Z \prod_{\substack{j=1 \\ j \neq i}}^r F_j\frac{\partial F_i}{\partial Y}\right) dY \\
+\left(\sum_{i=1}^r \alpha_i Z\prod_{\substack{j=1 \\ j \neq i}}^r F_j \frac{\partial F_i}{\partial Z}-d\prod_{j=1}^r F_j\right) dZ,
\end{multline*}
where $F_i$ is the projectivization of $f_i$, $ 1 \leq i \leq r$,  and $d=\sum_{i=1}^r \alpha_i \deg(F_i)$.
\end{pro}
\begin{proof}
It is straightforward to check that $\prod_{i=1}^r f_i^{\alpha_i}$ is a first integral the vector field obtained by the restriction of $\Omega_{\mathcal{X}}$ to the affine plane. So, it only remains to prove that $\Omega_{\mathcal{X}}$ is reduced.

Let ${\mathcal{X}}'$ be the map which sends every element $\bar{\beta}=(\beta_1, \beta_2, \ldots,\beta_r)\in \mathbb{R}^r$ to the  vector field ${\mathcal{X}}'(\bar{\beta})$ of $\mathbb{CP}^2$ defined by the homogeneous 1-form
\begin{multline*}
\Omega(\beta_1, \beta_2, \ldots,\beta_r):=\left(\sum_{i=1}^r \beta_i Z \prod_{\substack{j=1 \\ j \neq i}}^r F_j \frac{\partial F_i}{\partial X}\right) dX+\left(\sum_{i=1}^r \beta_iZ \prod_{\substack{j=1 \\ j \neq i}}^r F_j\frac{\partial F_i}{\partial Y}\right) dY \\
+\left(\sum_{i=1}^r \beta_i Z \prod_{\substack{j=1 \\ j \neq i}}^r F_j\frac{\partial F_i}{\partial Z}-\left(\sum_{i=1}^r \beta_i \deg(F_i)\right)\prod_{j=1}^r F_j\right) dZ.
$$
\end{multline*}

Denote by $\mathbb{Q}^+$ the set of positive rational numbers. It is straightforward to check that, for each $\bar{\beta}\in (\mathbb{Q}^+)^r$,  the vector field $\mathcal{X}'(\bar{\beta})$ has $\prod_{i=1}^r F_i^{n_i}/Z^m$ as a first integral, where $n_i=e\beta_i$, $1 \leq i \leq r$, $e$ is the least common multiple of the denominators of the irreducible expressions of the rational numbers $\beta_1, \beta_2, \ldots, \beta_r$ and $m=\sum_{i=1}^r n_i\deg(F_i)$.

By \cite[Lemma 1 (iv)]{c-p-r-2}, the pencil of curves defined by the equations
$$\left\{\lambda_1 \prod_{i=1}^r F_i^{n_i} +\lambda_2 Z^m=0 \;\; \vert \;\;(\lambda_1:\lambda_2)\in \mathbb{CP}^1\right\}$$
has exactly two elements which are not integral (reduced and irreducible) curves. These curves are those with equations $\prod_{i=1}^r F_i^{n_i}=0$ and $Z^m=0$. We must recall that this result holds when the curves $F_i=0$ have only one place at infinity. Otherwise this number may be larger \cite{F2012}.

Now, $\Omega(\bar{\beta})$ is reduced for all $\bar{\beta}\in (\mathbb{Q}^+)^r$ because, otherwise, the formula relating the degrees of the pencil and the form, and the factorization of the remarkable curves \cite[Lemma 1.2]{GarciaZamora}, does not hold.

Consider now the function $ \omega({\bar{\beta}})$ which maps any element $\bar{\beta} \in \mathbb{R}^r$ to the affine $1$-form $\omega({\bar{\beta}})=a_{\bar{\beta}}(x,y)dx+b_{\bar{\beta}}(x,y)dy$, where
$$a_{\bar{\beta}}(x,y) :=\sum_{i=1}^r \beta_i \prod_{\substack{j=1 \\ j \neq i}}^r f_j\frac{\partial f_i}{\partial x}\;\mbox{ and }\;b_{\bar{\beta}}(x,y) : =\sum_{i=1}^r \beta_i \prod_{\substack{j=1 \\ j \neq i}}^r f_j\frac{\partial f_i}{\partial y}.$$

Regarding $a_{\bar{\beta}}(x,y)$ and $b_{\bar{\beta}}(x,y)$ as polynomials (with coefficients in a suitable ring) in the variable $x$ (respectively $y$), we can compute the resultant $\mathrm{Res}_x(a_{\bar{\beta}},b_{\bar{\beta}})$ (respectively $\mathrm{Res}_y(a_{\bar{\beta}},b_{\bar{\beta}})$). The fact that $\Omega(\bar{\beta})$ is reduced for vectors $\bar{\beta}\in (\mathbb{Q}^+)^r$ implies the same fact for $ \omega({\bar{\beta}})$, which proves that $\mathrm{Res}_x(a_{\bar{\beta}},b_{\bar{\beta}}) \in \mathbb{C}[\beta_1, \beta_2,\ldots,\beta_r][y]$ (respectively, $\mathrm{Res}_y(a_{\bar{\beta}},b_{\bar{\beta}}) \in \mathbb{C}[\beta_1, \beta_2, \beta_2,\ldots,\beta_r][x]$) is a nonzero polynomial. As a consequence, setting $\mathbb{R}^+$ the set of positive real numbers, the $1$-form $ \omega({\bar{\alpha}})$ is reduced for all values $\bar{\alpha} \in (\mathbb{R}^+)^r$ such that $\mathrm{Res}_x(a_{\bar{\alpha}},b_{\bar{\alpha}}) \neq 0$ and $\mathrm{Res}_y(a_{\bar{\alpha}},b_{\bar{\alpha}}) \neq 0$. This proves that the $1$-form $\Omega_{\mathcal{X}}$ is reduced for generic exponents $\alpha_1, \alpha_2, \ldots,\alpha_r \in \mathbb{R}^+$ and concludes the proof.
\end{proof}

Now we state one of our main results, which determines the configuration $\mathcal{E}_{\infty}(\mathcal X)$ corresponding to projective vector fields $\mathcal X$ having a DPWAI first integral with generic exponents.

\begin{teo}\label{teorema2}
Assume that $\mathcal{X}$ is the complex projectivization of a  polynomial vector field $\X$ having a DPWAI first integral as above whose exponents $\alpha_1, \alpha_2, \ldots,\alpha_r$ are generic. Then the following equality of configurations holds:
$$
\mathcal{E}_{\infty}(\mathcal X)=\bigcup_{i=1}^r  \left(\mathcal{BP}({\mathcal P}_i)\cup \{S_i\} \cup {\mathcal J}_i\right),
$$
where,  for $1 \leq i \leq r$, the $S_i$'s are the  points defined in Proposition \ref{previa} and each ${\mathcal J}_i=\{R_{1i},R_{2i},\ldots\}$ is an infinite chain of infinitely near points such that $R_{1i}$ (respectively, $R_{ji}$, $j \geq 2$) is a point that
belongs to the exceptional divisor obtained by blowing-up $S_i$ (respectively, $R_{j-1,i}$).
Moreover, the proximity graph of each chain $\{S_i\}\cup \mathcal{J}_i$ is  determined by the irrational number $\delta_i/\alpha_i$ and named ${\bf Prox}(\delta_i/\alpha_i)$, where
$$\delta_i : =\sum_{\substack{j=0 \\ j \neq i}}^r \alpha_j\rho_{ji},$$
and $\rho_{ji}$ are the integers $\sum_Q (C_j,C_i)_Q$ defined in (\ref{ro}).
\end{teo}

\begin{proof}
Consider the vector $\bar{\alpha}=(\alpha_1, \alpha_2, \ldots,\alpha_r) \in (\mathbb{R}^{+})^r$ given by some generic elements $\{\alpha_i\}_{i=1}^r$. Let $B:= B_{\bar{\alpha}}=\{\bar{\beta}_n=(\beta_{1n}, \beta_{2n}, \ldots, \beta_{rn})\}_{n=1}^{\infty}$ be a sequence of vectors in $(\mathbb{Q}^+)^r$ such that $\lim_{n\rightarrow \infty} \bar{\beta}_n=\bar{\alpha}$ and, for each positive integer $n$, denote by $\mathcal{X}^B_n$ the vector field of $\mathbb{CP}^2$ given by the homogeneous 1-form $\Omega(\bar{\beta}_n)$ defined as in the proof of Proposition \ref{prop}. Notice that each vector field $\mathcal{X}^B_n$ has the following rational first integral:
\[
\frac{\prod_{i=1}^n F_i^{e_{n}\beta_{in}}}{Z^{m_{\bar{\beta}_n}}},
\]
where $e_{n}$ is the least common multiple of the denominators of the irreducible expressions of $\beta_{1n}, \beta_{2n}, \ldots,\beta_{rn}$ and
\[
m_{\bar{\beta}_n} := \sum_{i=1}^r e_{n}\beta_{in}\deg(F_i).
\]
Hence, the restriction of $\mathcal{X}^B_n$ to the affine chart given by the complement of the line $Z=0$ has a WAI first integral. Notice that the dicritical configurations of the vector fields $\mathcal{X}^B_n$, $n \geq 1$, contain the configuration of infinitely near points of $\mathbb{CP}^2$: $\bigcup_{i=1}^r  \left(\mathcal{BP}({\mathcal P}_i)\cup \{S_i\}\right)$ (see Remark \ref{NOTA}), and also that $$\mathcal{D}(\mathcal{X}^B_n)\cap \mathbb{CP}^2=\mathcal{E}_{\infty}(\mathcal X)\cap \mathbb{CP}^2$$
by the proof of Proposition \ref{prop}.

Let ${\mathcal P}^B_n$ be the pencil of curves defined by the equations
$$
\lambda_1 \prod_{i=1}^n F_i^{e_{n}\beta_{in}}+\lambda_2 Z^{m_{\bar{\beta}_n}}=0,
$$
where $(\lambda_1:\lambda_2)\in \mathbb{CP}^1$. By Proposition \ref{previa}, the configuration of base points of this pencil, which coincides with $\mathcal{D}({\mathcal X}^B_n)$ \cite[Corollary 1]{FGM}, is given by
 $$\bigcup_{i=1}^r  \left(\mathcal{BP}({\mathcal P}_i)\cup \{S_i\} \cup {\mathcal L}^B_{in}\right),$$
where $\{S_i\}\cup {\mathcal L}^B_{in}$ is the configuration of base points of the (local) pencil (at $S_i$) defined by the non-zero linear combinations of $u^{e_{n}\beta_{in}}$ and $t^{\delta^B_{in}}$, where $u=0$ (respectively, $t=0$) is a local equation of the strict transform of $C_i$ at $S_i$ (respectively, the exceptional divisor containing $S_i$) and
$$
\delta^B_{in}:=e_{n}\sum_{\substack{j=0 \\ j \neq i}}^r  \beta_{jn} \rho_{ji}.
$$
Notice that, since $(u,t)$ is a regular system of parameters of the local ring at $S_i$ \cite[Lemma 1]{c-p-r-2}, the proximity graph of the configuration $\{S_i\}\cup {\mathcal L}^B_{in}$ is $${\bf Prox}\left(\frac{\delta^B_{in}}{e_{n}\beta_{in}}\right).
$$

Now consider the configuration of infinitely near points of $\mathbb{CP}^2$ $$\mathcal{B}:=\bigcup_{i=1}^r  \left(\mathcal{BP}({\mathcal P}_i)\cup \{S_i\}\right).$$
The reduction of singularities of $\mathcal{X}$ shows that $ \mathcal{B} \subseteq \mathcal{E}_{\infty}(\mathcal{X})$ because $\bar{\alpha}$ is taken to be generic.  In addition, $S_1, S_2, \ldots, S_r$ are the unique common points of $\mathcal{E}_{\infty}(\mathcal{X})$ and the surface $V$ obtained by blowing-up the points in $\cup_{i=1}^r  \mathcal{BP}({\mathcal P}_i)$. This statement is a consequence of the forthcoming Lemma \ref{LEMA}, \cite[Lemma 1]{FGM} (where it is proved that under our conditions the operations ``blowing-up" and ``taking associated $1$-forms" commute) and the fact that the local equation of the strict transform of $\mathcal{X}$ at any point of $V$ is the limit (when $n$ tends to infinity) of the local equations of the strict transforms of $\mathcal{X}_n^B$.

To finish our proof and give a complete description of $\mathcal{E}_{\infty}(\mathcal{X})$. We notice that by \cite[Lemma 1]{FGM} the $1$-form $$
e_{n}\beta_{in}t\;du-\delta_{in}^Bu\; dt, \;\;\; 1 \leq i\leq r,
$$ locally defines the strict transform of ${\mathcal X}_n^B$ at $S_i$, where $u=0$ is a local equation of the strict transform of $C_i$ at $S_i$ and $t=0$ is a local equation of the exceptional divisor. Then, taking limits, the $1$-form $$t\;du-\frac{\delta_i}{\alpha_i}\; dt$$ defines the strict transform of $\mathcal{X}$ at $S_i$.  From this local expression of the vector field, it is straightforward to deduce  that the configuration of infinitely near points to $S_i$  belonging to ${\mathcal E}_{\infty}({\mathcal X})$ is an infinite chain whose proximity graph is $\bold{Prox}(\delta_i/\alpha_i)$. In fact, this chain is infinite because Condition (3) in Definition \ref{DPWAI} implies the irrationality of  $\delta_i/\alpha_i$ for all $i\in \{1, 2, \ldots,r\}$.
\end{proof}

We conclude this section by stating and proving the above used Lemma \ref{LEMA}, which allows us to prove that, with the notation as in the proof of the above theorem, the surface $V$ obtained by blowing-up the points in $\cup_{i=1}^r  \mathcal{BP}({\mathcal P}_i)$ has no points in ${\mathcal E}_{\infty}({\mathcal X})$ different from the $S_i$'s, $1 \leq i \leq r$.

The proof of Lemma \ref{LEMA} will require to use some objects of algebraic geometry which we briefly summarize for convenience of the reader. We will consider divisors on a surface $W$ and the {\it Ner\'on-Severi group} of $W$, NS$(W)$, which is the group of numerical (equivalently, linear, in our case) equivalence classes $[C]$ of divisors $C$ on $W$. Recall that two divisors in $W$ are linearly equivalent whenever its difference is principal \cite{hart}. One can transform this group into an $\mathbb{R}$-linear space by using tensorial product, NS$(W)\otimes \mathbb{R}$, and then the {\it cone of curves} of $W$ is defined as
\[
\mathrm{NE}(W) = \left\{ \sum a_i [C_i] \; \vert \; C_i \; \mbox{is a reduced and irreducible curve of $W$, $a_i\in \mathbb{R}$ and $a_i \geq 0$}
\right\}.
\]

The topological closure of $\mathrm{NE}(W)$ in NS$(W)\otimes \mathbb{R}$, $\overline{\mathrm{NE}(W)}$, is the so-called {\it closed cone of curves} of $W$. Extremal rays of cones  $\overline{\mathrm{NE}(X)}$ of algebraic varieties $X$ are crucial objects in the model minimal program in algebraic geometry \cite{mori82, kawamata84, hacon10}.

\begin{lem}
\label{LEMA}
Keep the notation as in Theorem \ref{teorema2} and its proof. Let $\pi: W \rightarrow \mathbb{CP}^2$ be the composition of the blowing-ups centered at the points of the dicritical configuration, $\mathcal{D}(\mathcal{X}_n^B)$, of the vector field of $\mathbb{CP}^2$ defined by $\Omega(\bar{\beta}_n)$. Consider the set $\{Q_1, Q_2, \ldots,Q_r\}$ of maximal points of $\mathcal{D}(\mathcal{X}_n^B)$ and the union $\Gamma=\bigcup_{i=1}^r \tilde{E}_{Q_i}$ of the strict transforms on $W$ of the exceptional divisors obtained by blowing-up the points $Q_i$, $1\leq i \leq r$.

Then, the singularities of the strict transform of $\mathcal{X}_n^B$ on $W$ which belong to the exceptional locus of $\pi$, but not to $\Gamma$, are simple and the local equation of $\mathcal{X}_n^B$ at any of them has the form $$at_2dt_1+bt_1dt_2,$$ where $a,b\in \mathbb{Q}^+\cup \{0\}$, $a + b \neq 0$ and, $t_1=0$ (respectively, $t_2=0$) is the local equation, at the singularity, of the strict transform of an irreducible exceptional divisor (respectively, the line at infinity $Z=0$).
\end{lem}

\begin{proof} First of all notice that, by  Proposition \ref{previa}, the dicritical configuration of $\mathcal{X}_n^B$ and the base points configuration of  $\mathcal{P}_n^B$ coincide; i.e., $\mathcal{D}(\mathcal{X}_n^B)=\mathcal{BP}(\mathcal{P}_n^B)$.

Let $H$ be an arbitrary irreducible component of a curve of the pencil $\mathcal{P}_n^B$, different from $Z=0$, and denote by $\tilde{H}$ its strict transform on $W$. To prove the lemma, it is enough to show that, outside $\Gamma$, $\tilde{H}$ has no singularity on the exceptional locus of $\pi$.

We can assume $H \neq C_i$ for all $i \in \{1, 2, \ldots, r\}$  because $\pi$ is a common resolution of the singularities at infinity of these curves (see \cite[Lemma 1 (iii)]{c-p-r-2} and \cite{pa2, Mo}).

Notice that the self-intersection of $\tilde{H}$ cannot be negative because, in this case, the class of $\tilde{H}$ in $NS(W)\otimes \mathbb{R}$ would generate an extremal ray of the closed cone of curves of $W$  \cite[Lemma 1.22]{kollar}, and this is not possible by \cite[Theorem 3]{c-p-r-2}. Hence, $\tilde{H}^2\geq 0$.

The pencil $\mathcal{P}_n^B$ defines a rational map $\mathbb{CP}^2\cdots \rightarrow \mathbb{CP}^1$, that is a morphism from an open set of $\mathbb{CP}^2$ to $\mathbb{CP}^1$ which cannot be extended to any larger open set. The elimination of its indeterminacies induces a morphism $f:W\rightarrow \mathbb{CP}^1$ contracting to a point the strict transform of any curve in the pencil \cite[Theorem II.7]{beau}. In particular, it contracts $\tilde{H}$. The morphism $f$ is defined by global sections of the invertible sheaf \cite{hart} associated to the divisor $D$ on $W$ given by the strict transform of a general curve of the pencil $\mathcal{P}_n^B$. Notice that $D^2=0$ because the strict transforms on $W$ of two different general curves of the pencil do not meet. Then, $$D\cdot \tilde{H}=0$$ (and $\tilde{H}^2\geq 0$ by the previous paragraph). Hence $\tilde{H}$ must be linearly equivalent to $D$ \cite[Lemma III.9]{beau}.

To conclude the proof, we take into account that the unique exceptional irreducible divisors which are not contracted by the morphism $f$ are the irreducible components of $\Gamma$ \cite[Lemma 1 (iii)]{c-p-r-2}, and this implies that, if $E$ is the strict transform of an irreducible exceptional divisor, then $$\tilde{H}\cdot E=D\cdot E>0$$ if and only if $E$ is contained in $\Gamma$. This proves that $\tilde{H}$ can only have singularities on $\Gamma$ and the lemma.

\end{proof}

\section{The algorithm}
\label{Alg}
In this last section we state the previously mentioned algorithm which computes a first integral of planar polynomial vector fields having a DPWAI first integral with generic exponents.

Darboux proved in \cite{dar} that if a polynomial vector field $\mathbf{X}$ of degree $d$ has at least $\binom{d + 1}2+1$ invariant algebraic curves, then it has a Darboux first integral, which can be computed using these invariant algebraic curves.  This result was improved in \cite{CL2} (see also \cite{Che}). Next we state the Darboux theorem.

\begin{teo}\label{tDar}
Suppose that a polynomial differential system $\mathbf{X}$ as in \eqref{e1} of degree $d$ admits $r$ irreducible invariant algebraic curves $f_i(x,y)=0$ with respective cofactor $k_i(x,y)$, $1 \leq i \le r$. Then:
\begin{enumerate}
\item[{\rm(a)}] There exist $\lambda_i\in \mathbb{C}$, not all zero, such that
    \begin{equation}
    \label{gege}
    \mathop{\sum}\limits_{i=1}^r\lambda_i k_i(x,y)=0
   \end{equation}
     if and only if the function
\begin{equation}\label{idar}
H=f_1^{\la_1}\cdots f_p^{\la_r}
\end{equation}
is a first integral of the system $\mathbf{X}$.

\item[\rm(b)] If $r=\binom {d+1}2+1$, then there exist $\lambda_i\in\C$, not all zero, such that $\mathop{\sum}\limits_{i=1}^r\lambda_i k_i(x,y)=0$.
\end{enumerate}
\end{teo}

 As a consequence of the above result, finding invariant algebraic curves is an important tool in the study of Darboux integrability; notice also that it is  a very hard problem.

Next we present Theorem \ref{teorema3} which, together with some previous results, allows us to state our algorithm. This algorithm provides enough invariant curves to apply Theorem \ref{tDar} and determine a Darboux first integral in case the input we supply is a polynomial vector field having a DPWAI first integral with generic exponents.

The mentioned theorem will use a cluster $({\mathcal K},{\bf m}_{{\mathcal K}})$ attached to any finite chain ${\mathcal K}$ of infinitely near points of $\mathbb{CP}^2$. Recall that a configuration ${\mathcal K} = \{Q_1,Q_2, \ldots, Q_s\}$ is a chain if $Q_i$, $i > 1$, belongs to the exceptional divisor created by $Q_{i-1}$. The sequence of positive integers $\bold{m}_{\mathcal K} := (m_{Q})_{Q\in {\mathcal K}}$ is defined as follows: $m_Q=1$ if $Q$ is the maximal point of ${\mathcal K}$ and $m_Q=\sum m_R$ otherwise, where the sum is taken over the set of points $R$ in ${\mathcal K}$ which are proximate to $Q$. Also, given an arbitrary configuration ${\mathcal C}$ and any point $Q\in {\mathcal C}$, we shall denote by ${\mathcal C}_Q$ the finite chain defined by those points $R$ in ${\mathcal C}$ such that $Q$ is infinitely near to $R$.

\begin{teo}
\label{teorema3}
Let $\mathcal X$ be a projective vector field and $\alpha_1, \alpha_2, \ldots,\alpha_r$ real numbers as in Theorem \ref{teorema2}. Keep the above notation and consider the clusters
${\mathcal K}_i:=({\mathcal C}_i,{\bf m}_{{\mathcal C}_i})$, $1 \leq i \leq r$, where ${\mathcal C}_i=\{L_{i0},L_{i1},\ldots, L_{i{\ell}_i}:=S_i\}:={\mathcal C}_{S_i}=\mathcal{BP}({\mathcal P}_i)\cup \{S_i\}$. Set ${\bf m}_{{\mathcal C}_i}=(m_Q)_{Q\in {\mathcal C}_i}$.
Then the following equalities hold:
\begin{itemize}
\item[(i)] $[\deg(C_i)]^2-\sum_{j=0}^{\ell_i} m_{L_{ij}}^2=-1$, and
\item[(ii)] $\deg(C_i)=\sum_{j=0}^{\kappa_i} m_{L_{ij}}$,
\end{itemize}
where $\kappa_i$ denotes the maximum index $j$ such that the strict transform of the line at infinity passes through $L_{ij}$. Moreover, $C_i$ is the unique curve in the linear system ${\mathcal L}_{\sum_{j=1}^{\kappa_i} m_{L_{ij}}}({\mathcal K}_i)$.
\end{teo}

\begin{proof}

Let $U$ be the surface obtained after blowing-up the points in $\mathcal{BP}({\mathcal P}_i)$. Consider the line at infinity $L$ and two different general curves $\Delta_1$ and $\Delta_2$ of the pencil ${\mathcal P}_i$. Then, their strict transforms on $U$, $\tilde{L}$, $\tilde{\Delta}_1$ and $\tilde{\Delta}_2$, do not meet. Therefore $\tilde{\Delta}_1\cdot \tilde{\Delta}_2=0$ and $\tilde{\Delta}_1\cdot \tilde{L}=0$. Both equalities prove, respectively, the equalities (i) and (ii) after noticing that $\tilde{\Delta}_1$ and $\tilde{\Delta}_2$ are linearly equivalent to the strict transform of $C_i$.

We conclude the proof by noticing that our last assertion follows from (i) and B\'ezout Theorem \cite[I.9 (a)]{beau}.
\end{proof}

Now we present our algorithm which, applied to a polynomial vector field, computes  candidates to be the polynomials $f_1, f_2, \ldots, f_r$ appearing in a DPWAI first integral. Theorems \ref{teorema2} and \ref{teorema3} prove that when the input has a DPWAI first integral with generic exponents, the output will be the mentioned polynomials $f_i$.

We will need the following notation: given an arbitrary configuration ${\mathcal C}$ we define, for each maximal point $Q$ of ${\mathcal C}$, the integer $I_{Q}({\mathcal C}) :=d_{Q}({\mathcal C})^2-\sum_{P\in {{\mathcal C}}_Q} m_P^2$, where ${\bf m}_{{\mathcal C}_Q}=(m_P)_{P\in {\mathcal C}_Q}$ and $d_{Q}({\mathcal C}):=\sum m_P$, the sum being taken over the points $P$ in $\mathcal C_Q$ such that the strict transform of the line at infinity passes through $P$.

\begin{alg}\label{alg1}

{\rm

\begin{itemize} $\;$

\item \emph{Input:} An arbitrary polynomial vector field $\mathbf{X}$.

\item \emph{Output:} Either a finite set $\{f_i(x,y)\}_{i=1}^r$ of polynomials in two variables which are candidates for applying Theorem \ref{tDar} and obtaining a Darboux first integral, or $0$.

\end{itemize}
\medskip

\begin{enumerate}
 \item Compute an homogeneous 1-form defining the complex projectivization $\mathcal{X}$ of $\mathbf{X}$.
 \item Compute the set $\Omega'$ consisting of the points $Q$ in the singular configuration ${\mathcal S}({\mathcal X})$ which are infinitely near to a point of the line at infinity.
 \item Let $Q_1, Q_2, \ldots,Q_{\ell}$ be the maximal points of $\Omega'$. For every $i\in \{1,2,\ldots,\ell\}$ compute the maximal
 configuration $\Omega^i$ of points $P$ infinitely near to $Q_i$ satisfying the following conditions:
 \begin{itemize}
 \item[(a)] $P$ is free,
 \item[(b)] $P\in {{\mathcal E}_{\infty}({\mathcal X})}$ (that is, $P$ is a simple singularity of the strict transform of the vector field $\mathcal X$ whose associated quotient of eigenvalues is a positive irrational number),
 \item[(c)] $I_P({{\mathcal E}_{\infty}({\mathcal X})})\geq -1$.
 \end{itemize}
 If $\Omega^i$ is empty for some $i\in \{1,2,\ldots,\ell\}$, then return $0$. Else, define $\Omega:=\Omega' \cup \Omega^1\cup \cdots \cup \Omega^{\ell}$ and go to Step (4).

\item Let $M=\{S_1, S_2, \ldots, S_r\}$ be the set of maximal points of $\Omega$. If  $I_{S_i}({{\mathcal E}_{\infty}({\mathcal X})})\neq -1$ for some $i \in \{1,2,\ldots,r\}$ then return 0. Else go to Step (5).

\item If the linear systems ${\mathcal L}_{d_{S_i}}({\Omega}_{S_i},\bold{m}_{\Omega_{S_i}})$ have projective dimension $0$ for all $i\in \{1, 2, \ldots, r\}$ (where $d_{S_i}:=d_{S_i}({{\mathcal E}_{\infty}({\mathcal X})})$ ), then return $$\{F_1(x,y,1), F_2(x,y,1), \ldots, F_r(x,y,1)\},$$ $F_i(X,Y,Z)$ being an homogeneous polynomial defining the unique curve in $${\mathcal L}_{d_{S_i}}({\Omega}_{S_i},\bold{m}_{\Omega_{S_i}}).$$ Else, return $0$.

\end{enumerate}
}
\end{alg}

Our procedure to decide about DPWAI integrability of a vector field $\X$ has two steps. First we run Algorithm \ref{alg1} with input $\mathbf{X}$ and, when the output is not $0$, we get $r$ candidates to be invariant algebraic curves of $\mathbf{X}$  given by equations $f_i=0$, $1 \leq i \leq r$. When these curves are invariant by $\mathbf{X}$, we compute their cofactors
\[
k_i=\frac{p\pd {f_i}x+q\pd{f_i}y}{f_i},
\]
and then, we   check whether there exist values $\lambda_i\in\mathbb{R}^+$, $1 \leq i \leq r$, satisfying Equality (\ref{gege}). Notice that we only need to solve a homogeneous linear system of equations and  this linear system has  $\binom {d+1}2$ equations, corresponding with the number of monomials of a polynomial of degree $d-1$ in two variables, and $r$ unknowns, say the $\lambda_i$'s. If such values $\lambda_i$ exist, then \eqref{idar} is a first integral of the system $\mathbf{X}$.

When the input $\X$ has a DPWAI first integral with generic exponents, we also obtain its extended resolution of singularities over the line at infinity. Otherwise,
$\X$  has not a DPWAI first integral with generic exponents, the output of Algorithm \ref{alg1} could be $0$ or some not necessarily invariant curves by $\mathbf{X}$. However, the output of Algorithm \ref{alg1} could also provide enough invariant curves and then, we would obtain a DPWAI first integral by means of Theorem \ref{tDar}.\\

We conclude this paper with an example where we detail our procedure to determine a Darboux first integral of the mentioned class of polynomial vector fields.

\begin{exam}
{\rm
Consider the polynomial vector field $${\bf X}=a(x,y)dx+b(x,y)dy,$$ where
$$a(x,y)=(3 +4 \pi ) x^6 y^2+(3 +\sqrt{2}+4 \pi ) x^7 +4 \pi  x^3 y^3 +(\sqrt{2} +4 \pi)  x^4 y-3 x^2 y^3 -(3 +\sqrt{2}) x^3 y -\sqrt{2} y^2$$
and
$$b(x,y)=2 \sqrt{2} x^7 y +(1+2 \sqrt{2}) x^4 y^2 +x^5 -(2 \sqrt{2} +\pi)  x^3 y^2 -\pi  x^4 -(1+2 \sqrt{2}+\pi ) y^3 -(1+\pi)  x y.$$




Algorithm \ref{alg1} gives rise to a configuration of infinitely near points of $ \mathbb{CP}^2$, $\Omega=\{P_i\}_{i=1}^{18}$, which has 3 maximal points: $$S_1=P_9, \; \;S_2=P_{13}\;\; \mathrm{and}\;\; S_3=P_{18}.$$
The proximity graph of the configuration is displayed in Figure \ref{F2}. Moreover, the multiplicity sequences are
$${\bf m}_{\Omega_{S_1}}=(3,1,1,1,1,1,1,1,1),$$
$${\bf m}_{\Omega_{S_2}}=(2,1,1,1,1,1,1),$$
$${\bf m}_{\Omega_{S_3}}=(1,1,1,1,1)$$
and, since the strict transforms of the line at infinity pass through $P_1$ and $P_2$, $d_{S_1}=4$, $d_{S_2}=3$, $d_{S_3}=2$ and $I_{S_i}=-1$ for all $i\in \{1,2,3\}$. In addition the algorithm allows us to determine that
$${\mathcal L}_{d_i}(\Omega_{S_i},{\bf m}_{\Omega_{S_i}})=\{C_i\},$$
where $C_1$ (respectively, $C_2$, $C_3$) is the projective curve (having only one place at infinity) with equation $X^4-YZ^3=0$ (respectively, $X^3+YZ^2=0$, $Y^2+XZ=0$).
In fact, Algorithm \ref{alg1} returns the set $\{f_1,f_2,f_3\}$, where $f_1(x,y)=x^4-y$, $f_2(x,y)=x^3+y$ and $f_3(x,y)=y^2+x$.

Now, applying Theorem \ref{tDar}, one obtains that $f_1^{\pi} f_2 f_3^{\sqrt{2}}$ is a DPWAI first integral of ${\bf X}$.

\begin{figure}[ht!]
\centering
\setlength{\unitlength}{0.5cm}
\begin{picture}(26,10)
\qbezier[20](11,1)(9,2)(10,4)


\put(11,1){\circle*{0.3}}\put(10.2,0.8){\tiny{$P_1$}}
\put(11,1){\line(0,1){1}}

\put(11,2){\circle*{0.3}}\put(11.1,1.9){\tiny{$P_2$}}
\put(11,2){\line(0,1){1}}

\put(11,3){\circle*{0.3}}\put(10.1,2.8){\tiny{$P_3$}}
\put(11,3){\line(-1,1){1}}

\put(10,4){\circle*{0.3}}\put(10.3,3.8){\tiny{$P_4$}}
\put(10,4){\line(-1,1){1}}

\put(9,5){\circle*{0.3}}\put(9.3,4.8){\tiny{$P_5$}}
\put(9,5){\line(-1,1){1}}

\put(8,6){\circle*{0.3}}\put(8.3,5.8){\tiny{$P_6$}}
\put(8,6){\line(-1,1){1}}

\put(7,7){\circle*{0.3}}\put(7.3,6.8){\tiny{$P_7$}}
\put(7,7){\line(-1,1){1}}

\put(6,8){\circle*{0.3}}\put(6.3,7.8){\tiny{$P_8$}}
\put(6,8){\line(-1,1){1}}

\put(5,9){\circle*{0.3}}\put(5.3,8.8){\tiny{$P_9=S_1$}}






\put(19,1){\circle*{0.3}}\put(19.2,0.8){\tiny{$P_{14}$}}
\put(19,1){\line(0,1){1}}

\put(19,2){\circle*{0.3}}\put(19.2,1.8){\tiny{$P_{15}$}}
\put(19,2){\line(0,1){1}}

\put(19,3){\circle*{0.3}}\put(19.2,2.8){\tiny{$P_{16}$}}
\put(19,3){\line(0,1){1}}

\put(19,4){\circle*{0.3}}\put(19.2,3.8){\tiny{$P_{17}$}}
\put(19,4){\line(0,1){1}}

\put(19,5){\circle*{0.3}}\put(19.2,4.8){\tiny{$P_{18}=S_3$}}






\put(11,3){\line(1,1){1}}

\put(12,4){\circle*{0.3}}\put(12.2,3.8){\tiny{$P_{10}$}}
\put(12,4){\line(1,1){1}}

\put(13,5){\circle*{0.3}}\put(13.2,4.8){\tiny{$P_{11}$}}
\put(13,5){\line(1,1){1}}

\put(14,6){\circle*{0.3}}\put(14.2,5.8){\tiny{$P_{12}$}}
\put(14,6){\line(1,1){1}}

\put(15,7){\circle*{0.3}}\put(15.2,6.8){\tiny{$P_{13}=S_2$}}





\end{picture}

\caption{Proximity graph of the configuration obtained applying Algorithm \ref{alg1}.}
\label{F2}
\end{figure}
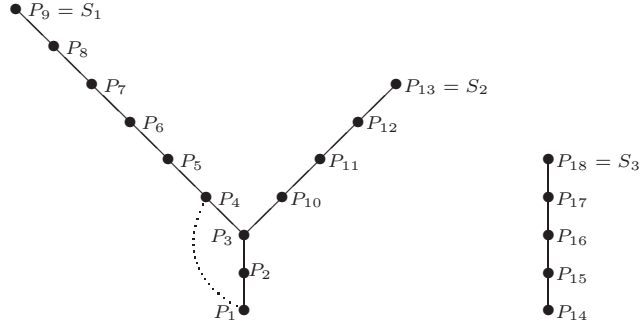

Finally, we notice that, with notations as in Theorem \ref{teorema2}, $$\delta_1=4+8\sqrt{2}, \; \delta_2=6\sqrt{2}+4\pi\;\; \mathrm{and}\;\; \delta_3=6+8\pi.$$
Therefore, the proximity graph of the extended reduction of singularities over the line at infinity, ${\mathcal E}_{\infty}({\mathcal X})$, is obtained from that in Figure \ref{F2} by adding three infinite chains  ${\mathcal J}_1$, ${\mathcal J}_2$ and ${\mathcal J}_3$ over $S_1$, $S_2$ and $S_3$ such that the proximity graph of ${\mathcal J}_1\cup \{S_1\}$ (respectively,  ${\mathcal J}_2\cup \{S_2\}$, ${\mathcal J}_3\cup \{S_3\}$) is ${\bf Prox}(\frac{4+8\sqrt{2}}{\pi})$ (respectively, ${\bf Prox}({6\sqrt{2}+4\pi})$, ${\bf Prox}\left(\frac{6+8\pi}{\sqrt{2}}\right)$). To complete our example we show, in Figure \ref{jota1}, the bottom part of the proximity graph of the chain ${\mathcal J}_1\cup \{S_1\}$, without labels.  A similar procedure provides the remaining chains.

\begin{figure}[ht!]
\centering
\setlength{\unitlength}{0.5cm}
\begin{picture}(26,14)
\qbezier[60](11,5)(9,8)(11,12)
\qbezier[20](11,4)(10,5)(11,6)
\qbezier[10](11,11)(10,11.5)(10,12)


\put(11,1){\circle*{0.3}}
\put(11,1){\line(0,1){1}}

\put(11,2){\circle*{0.3}}
\put(11,2){\line(0,1){1}}

\put(11,3){\circle*{0.3}}
\put(11,3){\line(0,1){1}}

\put(11,4){\circle*{0.3}}
\put(11,4){\line(0,1){1}}

\put(11,5){\circle*{0.3}}
\put(11,5){\line(0,1){1}}

\put(11,6){\circle*{0.3}}
\put(11,6){\line(0,1){1}}

\put(11,7){\circle*{0.3}}
\put(11,7){\line(0,1){1}}

\put(11,8){\circle*{0.3}}
\put(11,8){\line(0,1){1}}

\put(11,9){\circle*{0.3}}
\put(11,9){\line(0,1){1}}

\put(11,10){\circle*{0.3}}\put(11,10){\line(0,1){1}}

\put(11,11){\circle*{0.3}}
\put(11,11){\line(0,1){1}}

\put(11,12){\circle*{0.3}}

\put(11,12.5){$\vdots$}

\end{picture}
\caption{Proximity graph of the chain ${\mathcal J}_1\cup \{S_1\}$}
\label{jota1}
\end{figure}
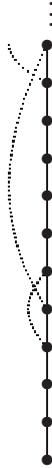

}
\end{exam}

\end{document}